\documentclass[12pt,numbers,sort&compress]{elsarticle}
\journal{}

\makeatletter
\def\ps@pprintTitle{%
 \let\@oddhead\@empty
 \let\@evenhead\@empty
 \def\@oddfoot{\hfill\thepage}%
 \let\@evenfoot\@oddfoot}
\makeatother

\usepackage{amsmath,amssymb,amsfonts,amsthm,graphicx}
\usepackage[bookmarksnumbered,colorlinks=true]{hyperref}
\usepackage[labelfont=bf]{caption}

\providecommand{\doi}[1]{\href{https://doi.org/#1}{DOI:#1}}
\usepackage{xurl} 
\renewcommand{\doi}[1]{%
 \href{https://doi.org/#1}{\nolinkurl{DOI:#1}}%
}

\usepackage{dsfont} 
\usepackage{enumitem} 
\usepackage{mathtools} 
\usepackage{geometry} 
\theoremstyle{plain}
\newtheorem{theorem}{Theorem}

\newtheorem{corollary}{Corollary}

\theoremstyle{definition}
\newtheorem{remark}{Remark}[section]

\newcommand{\N}{\mathbb{N}}
\newcommand{\R}{\mathbb{R}}

\newcommand{\EE}{\mathsf{E}}
\newcommand{\Var}{\mathsf{Var}}
\newcommand{\Cov}{\mathsf{Cov}}
\newcommand{\bb}[1]{\boldsymbol{#1}}
\newcommand{\rd}{\mathrm{d}}
\newcommand{\e}{\varepsilon}
\newcommand{\tr}{\mathrm{tr}}

\newcommand{\diag}{\operatorname{diag}}

\newcommand{\haf}{\operatorname{haf}}

\allowdisplaybreaks

\begin{document}

\begin{frontmatter}

\title{A proof of the strong Gaussian product inequality conjecture}

\author[a1]{Fr\'ed\'eric Ouimet\corref{mycorrespondingauthor}}
\author[a2]{Dylan Greaves}

\address[a1]{Universit\'e du Qu\'ebec \`a Trois-Rivi\`eres, Trois-Rivi\`eres, Canada}
\address[a2]{San Francisco, USA\vspace{-7mm}}

\cortext[mycorrespondingauthor]{Corresponding author. Email address: frederic.ouimet2@uqtr.ca}

\begin{abstract}
Let $\bb{X} = (X_1,\ldots,X_n)$ be a centered Gaussian vector, not necessarily nondegenerate. It is proved that, for every $\alpha_1,\ldots,\alpha_n > 0$,
\[
\EE\left[\prod_{i=1}^n |X_i|^{\alpha_i}\right] \geq \prod_{i=1}^n \EE\left[|X_i|^{\alpha_i}\right].
\]
When all marginal variances are positive, equality holds if and only if the coordinates are independent. This settles Frenkel's 18-year-old Gaussian product inequality (GPI) conjecture and, in fact, its later strengthening to arbitrary positive exponents. Through Frenkel's hafnian formulation, this result also provides a short proof of the 28-year-old real linear polarization constant conjecture, which was settled this year as a consequence of the strong polarization inequality. The main result leads to the exact real linear polarization constant, a sharp weighted product inequality for real linear functionals, spherical moment bounds, hafnian inequalities for positive-semidefinite matrices together with complete characterizations of their equality cases, logarithmic variance and covariance inequalities, a R\'enyi total-correlation certificate, and unconditional versions of mixed-sign Gaussian moment bounds previously conditional on the positive-exponent GPI.
\end{abstract}


\end{frontmatter}

\section{Introduction}\label{sec:intro}

The Gaussian product inequality (GPI) conjecture grew out of the real linear polarization constant problem in functional analysis. A norm-one linear functional on a Hilbert space attains its norm on the unit ball, but several norm-one functionals need not be simultaneously large at the same point. The real linear polarization constant measures the worst possible loss in their product when one common point must be used. For continuous real linear functionals $f_1,\ldots,f_n$ on a real Hilbert space $\mathcal{H}$, \citet{MR1636556} sought the optimal universal lower bound for the norm of the homogeneous polynomial $\prod_{i=1}^n f_i$. When $\dim(\mathcal{H}) \geq n$, their conjecture states that
\begin{equation}\label{eq:RLPC.conjecture}
\left\|\prod_{i=1}^n f_i\right\| \geq n^{-n/2}\prod_{i=1}^n \|f_i\|.
\end{equation}
Orthogonal functionals show that the factor $n^{-n/2}$ in \eqref{eq:RLPC.conjecture} cannot be improved.

\newpage
The corresponding problem over complex Hilbert spaces was solved by \citet{AriasdeReyna1998}, whereas the real problem remained open for nearly three decades. As detailed in Section~\ref{sec:consequences}, \citet{MartinezOrtegaMoreno2026} recently proved the strong polarization conjecture, a reciprocal-square strengthening of \eqref{eq:RLPC.conjecture}, and characterized all extremal configurations for the real linear polarization problem; \citet{GalicerOrtegaMorenoPinasco2026} subsequently proved a weighted extension.

In 2008, \citet{MR2385646} formulated the GPI conjecture, which states that for every centered Gaussian vector with covariance matrix $\Sigma = (\sigma_{ij})$ and every $m \in \N \equiv \{1,2,\ldots\}$,
\begin{equation}\label{eq:original.GPI}
\EE\left[\prod_{i=1}^n X_i^{2m}\right] \geq \prod_{i=1}^n \EE\left[X_i^{2m}\right].
\end{equation}
By Wick's formula, this takes the equivalent hafnian form $\haf\left(\Sigma[2m]\right) \geq \prod_{i=1}^n (2m - 1)!!\,\sigma_{ii}^m$, where $\Sigma[2m]$ is the matrix obtained by repeating each row index and its corresponding column index exactly $2m$ times. Frenkel observed that its validity for all positive integers $m$ would imply the sharp real linear polarization estimate in \eqref{eq:RLPC.conjecture}. The implication follows by representing the functionals as Gaussian linear forms, comparing their $2m$th moments with the radial moment of a standard Gaussian vector, and letting $m \to \infty$. \citet{MR2886380} subsequently proposed a stronger conjecture allowing arbitrary positive exponents:
\begin{equation}\label{eq:strong.GPI}
\EE\left[\prod_{i=1}^n |X_i|^{\alpha_i}\right] \geq \prod_{i=1}^n \EE\left[|X_i|^{\alpha_i}\right], \qquad \alpha_1,\ldots,\alpha_n > 0.
\end{equation}
The strong GPI conjecture \eqref{eq:strong.GPI} is proved in Theorem~\ref{thm:strong.GPI}; this establishes both the original GPI conjecture \eqref{eq:original.GPI} and Frenkel's equivalent hafnian formulation.

Before the present work, partial positive-exponent GPI results (for given dimensions $n \leq 5$, restricted classes of covariance matrices, specific exponents, or a combination thereof) were obtained by \cite{MR628759, MR2385646, MR4052574, MR4466643, MR4445681, MR4554766, MR4760098, MR4593134, MR4661091, MR4798604, MR4794515, Hirose2026}. Negative-exponent and mixed-sign extensions were obtained by \cite{MR3278931, MR4471184, MR4530374, MR4666255, MR4956548, MR5042459, LanOuimetSun2025}.
Closely related Wishart extensions appear in \cite{MR4538422, GOR2024EJP, arXiv:2409.14512}. 
A chronological description of these contributions (except for \cite{MR4956548,MR5042459,LanOuimetSun2025,Hirose2026}) is given by \citet{GOR2024EJP}. The proof of the strong GPI herein uses methods entirely different from those developed in the GPI works cited above, so their approaches are not reviewed here. Although the proof does not use the strong polarization inequality proved by \citet{MartinezOrtegaMoreno2026} or its weighted extension proved by \citet{GalicerOrtegaMorenoPinasco2026}, one specialization of its central matrix identity recovers the cancellation underlying their proofs; see Section~\ref{sec:connection}.

The rest of the paper is organized as follows. Section~\ref{sec:main.result} proves the strong GPI along with the characterization of equality in terms of independence. Section~\ref{sec:connection} details the connection with the works of \cite{MartinezOrtegaMoreno2026,GalicerOrtegaMorenoPinasco2026}. Corollaries of the strong GPI are collected in Section~\ref{sec:consequences}.

\section{Main result}\label{sec:main.result}

\begin{theorem}[Strong GPI]\label{thm:strong.GPI}
Let $\bb{X} = (X_1,\ldots,X_n)$ be a centered Gaussian vector, not necessarily nondegenerate. Then, for every $\alpha_1,\ldots,\alpha_n > 0$,
\begin{equation}\label{eq:thm:strong.GPI}
\EE\left[\prod_{i=1}^n |X_i|^{\alpha_i}\right] \geq \prod_{i=1}^n \EE\left[|X_i|^{\alpha_i}\right].
\end{equation}
Moreover, if $\Var(X_i) > 0$ for every $i$, then equality in \eqref{eq:thm:strong.GPI} holds if and only if $X_1,\ldots,X_n$ are independent.
\end{theorem}

\newpage
\begin{proof}
Let $\sigma_i^2 = \Var(X_i)$. If $\sigma_i = 0$ for some $i$, then $X_i = 0$ almost surely (a.s.) and both sides of \eqref{eq:thm:strong.GPI} are zero. One may therefore assume that $\sigma_i > 0$ for every $i$. Let $Y_i = X_i/\sigma_i$ and let $R$ be the covariance matrix of $\bb{Y}$. Then $R$ is a positive-semidefinite correlation matrix, and both sides of \eqref{eq:thm:strong.GPI} for $\bb{X}$ are the corresponding sides for $\bb{Y}$ multiplied by the same factor $\prod_{i=1}^n \sigma_i^{\alpha_i}$. Thus, after relabeling $\bb{Y}$ as $\bb{X}$, one may assume that $\bb{X} \sim \mathcal{N}(\bb{0},R)$, where $R$ is a positive-semidefinite correlation matrix.

Assume first that $R$ is positive definite. Fix $c_1,\ldots,c_n > 0$. For $\bb{x} \in (\R \setminus \{0\})^n$, define
\[
d_i(\bb{x}) = \frac{c_i}{x_i}, \qquad
\bb{d}(\bb{x}) = \bigl(d_1(\bb{x}),\ldots,d_n(\bb{x})\bigr)^{\top}, \qquad
F(\bb{x}) = \bb{x} - R\bb{d}(\bb{x}).
\]
For $\bb{y} \in \R^n$, consider
\[
W_{\bb{y}}(\bb{d}) = \frac{1}{2}\bb{d}^{\top}R\bb{d} + \bb{y}^{\top}\bb{d} - \sum_i c_i\log(|d_i|), \qquad
\nabla^2 W_{\bb{y}}(\bb{d}) = R + \diag(c_i/d_i^2) > 0.
\]
On each open orthant, $W_{\bb{y}}(\bb{d}) \to +\infty$ whenever $\|\bb{d}\| \to \infty$ or $\bb{d}$ approaches the boundary of the orthant. Thus, its restriction to each orthant has compact sublevel sets and is strictly convex, so it has a unique minimizer there. Index the orthants by $b$ and denote the corresponding minimizers by $\bb{d}_b = \bb{d}_b(\bb{y})$. Their stationarity equations are
\begin{equation}\label{eq:branch.equation}
\frac{c_i}{d_{b,i}} - (R\bb{d}_b)_i = y_i, \qquad i \in \{1,\ldots,n\}.
\end{equation}
Set $x_{b,i} = c_i / d_{b,i}$ with $\bb{x}_b = (x_{b,1},\ldots,x_{b,n})^{\top}$. The coordinatewise transformation $d_i \mapsto c_i/d_i$ maps each orthant onto itself, and \eqref{eq:branch.equation} is precisely $F(\bb{x}_b) = \bb{y}$. Consequently, for every $\bb{y} \in \R^n$, this equation has exactly one solution in each orthant.

Since the Hessian is nonsingular, the implicit function theorem shows that each $\bb{d}_b$, and hence each $\bb{x}_b$, depends smoothly on $\bb{y}$. It follows that the restriction of $F$ to each orthant is a $C^1$ diffeomorphism onto $\R^n$. Write $D_b = \diag(d_{b,i}^2/c_i)$. Then
\[
I + R D_b = D_b^{-1/2} (I + D_b^{1/2} R D_b^{1/2}) D_b^{1/2},
\]
so $\det(I + R D_b) > 0$. Set
\begin{equation}\label{eq:branch.weight}
w_b = \det\!\left(I + R\diag(d_{b,i}^2/c_i)\right)^{-1}.
\end{equation}
Since $DF(\bb{x}_b) = I + R D_b$, the quantity $w_b$ is positive and equals the Jacobian $\left|\det\!\left(\partial \bb{x}_b/\partial \bb{y}\right)\right|$.

Subtracting \eqref{eq:branch.equation} for two distinct branches $b$ and $b'$ gives, in matrix form,
\[
\left\{\diag\left(\frac{c_i}{d_{b,i}d_{b',i}}\right) + R\right\}(\bb{d}_b - \bb{d}_{b'}) = \bb{0}.
\]
Since $\bb{d}_b \neq \bb{d}_{b'}$ and $\diag(d_{b,i}d_{b',i}/c_i)$ is invertible, one may right-multiply the matrix in braces by this diagonal matrix and take determinants to obtain
\[
\det\!\left(I + R\diag\left(\frac{d_{b,i}d_{b',i}}{c_i}\right)\right) = 0, \qquad b \neq b',
\]
whereas for $b = b'$ the same determinant equals $w_b^{-1}$.

\newpage
Given any subset $S \subseteq [n] \equiv \{1,\ldots,n\}$, let $d_{b,S} = \prod_{i \in S}d_{b,i}$ and $c_S = \prod_{i \in S}c_i$, with $\prod_{\emptyset} = 1$, and let $\det(R_{\emptyset\emptyset}) = 1$. For $\Lambda = \diag(\lambda_i)$, the principal-minor identity
\[
\det(I + R \Lambda) = \sum_{S \subseteq [n]}\det(R_{SS})\prod_{i \in S} \lambda_i
\]
therefore shows that the square matrix
\[
U_{bS} = \sqrt{w_b}\sqrt{\frac{\det(R_{SS})}{c_S}}d_{b,S}
\]
satisfies $UU^{\top} = I$ (and thus $U^{\top} U = I$). Indeed, this is verified by the row inner products
\begin{equation*}
\begin{aligned}
(UU^{\top})_{bb'}
= \sum_{S \subseteq [n]} U_{bS} U_{b'S}
&= \sqrt{w_b w_{b'}} \sum_{S \subseteq [n]} \det(R_{SS}) \prod_{i \in S} \frac{d_{b,i} d_{b',i}}{c_i} \\
&= \sqrt{w_b w_{b'}} \det\!\left(I + R\diag\left(\frac{d_{b,i}d_{b',i}}{c_i}\right)\right) = \delta_{bb'},
\end{aligned}
\end{equation*}
where the last two equalities follow from the principal-minor identity and the previous determinant evaluations. Using $\det(R_{\emptyset\emptyset}) = 1$ and $R_{ii} = 1$, one applies $U^{\top} U = I$ to the column pairs $(S, S') = (\emptyset, \emptyset)$ and $(S, S') = (\{i\}, \{j\})$ and obtains, for every $\bb{y}$,
\begin{equation}\label{eq:determinantal.branch.identities}
\sum_b w_b = 1, \qquad \sum_b w_b\bb{d}_b\bb{d}_b^{\top} = \diag(c_1,\ldots,c_n).
\end{equation}

\emph{Change of variables and Jensen.} Let $\gamma_R$ be the $\mathcal{N}(\bb{0},R)$ law. Since $\bb{y} = \bb{x} - R\bb{d}(\bb{x})$ and $\bb{x}^{\top}\bb{d}(\bb{x}) = \sum_i c_i$,
\[
\bb{x}^{\top} R^{-1}\bb{x} + \bb{d}(\bb{x})^{\top} R\bb{d}(\bb{x}) = \bb{y}^{\top} R^{-1}\bb{y} + 2\sum_i c_i.
\]
With $\bb{c} = (c_1,\ldots,c_n)^{\top} \in (0,\infty)^n$, changing variables on all branches and then applying Jensen's inequality and \eqref{eq:determinantal.branch.identities} gives
\begin{equation}\label{eq:inverse.square.kernel.bound}
\begin{aligned}
K_R(\bb{c})
&\coloneqq \EE_{\gamma_R}\left[\exp\left(-\frac{1}{2}\sum_i\frac{c_i^2}{X_i^2}\right)\right] \\
&= e^{-\sum_i c_i} \, \EE_{\bb{Y} \sim \mathcal{N}(\bb{0},R)}\left[\sum_b w_b(\bb{Y})\exp\left(-\frac{1}{2}\bb{d}_b(\bb{Y})^{\top}(I - R)\bb{d}_b(\bb{Y})\right)\right] \\
&\geq e^{-\sum_i c_i} \, \EE_{\bb{Y} \sim \mathcal{N}(\bb{0},R)}\left[\exp\left(-\frac{1}{2}\sum_b w_b(\bb{Y})\bb{d}_b(\bb{Y})^{\top}(I - R)\bb{d}_b(\bb{Y})\right)\right] \\
&= e^{-\sum_i c_i}\exp\left(-\frac{1}{2}\tr\left((I - R)\diag(c_1,\ldots,c_n)\right)\right) = e^{-\sum_i c_i},
\end{aligned}
\end{equation}
where the last equality follows from $R_{ii} = 1$. Equality holds when $R = I$. The same kernel bound follows for positive-semidefinite correlation matrices. Indeed, set $R_{\e} = (1 - \e)R + \e I$. The matrix $R_{\e}$ is a positive-definite correlation matrix for $0 < \e < 1$. For the fixed $\bb{c} \in (0,\infty)^n$, extend the function $\bb{x} \mapsto \exp(-\frac{1}{2}\sum_i (c_i^2/x_i^2))$ by zero to the coordinate hyperplanes. The resulting function is bounded and continuous on $\R^n$. Since $\mathcal{N}(\bb{0},R_{\e})$ converges weakly to $\mathcal{N}(\bb{0},R)$ as $\e \downarrow 0$, one has $K_{R_{\e}}(\bb{c}) \to K_R(\bb{c})$, and the bound follows.

\newpage
\emph{Recovery of moments.} For $\alpha > 0$ and $x \in \R$,
\[
|x|^{\alpha} = \frac{2^{1 - \alpha/2}}{\Gamma(\alpha/2)}\int_0^{\infty} t^{\alpha - 1}\exp\left(-\frac{t^2}{2x^2}\right)\,\rd t,
\]
where the exponential factor is understood to be zero when $x = 0$. Applying this identity coordinatewise, using Tonelli's theorem, evaluating the resulting gamma integrals, and applying Legendre's duplication formula, one obtains
\begin{equation}\label{eq:moment.recovery.comparison}
\begin{aligned}
\EE\bigg[\prod_i |X_i|^{\alpha_i}\bigg]
&= \prod_i \frac{2^{1 - \alpha_i/2}}{\Gamma(\alpha_i/2)}\int_{(0,\infty)^n} \bigg(\prod_i t_i^{\alpha_i - 1}\bigg) K_R(\bb{t})\,\rd \bb{t} \\
&\geq \prod_i \frac{2^{1 - \alpha_i/2}}{\Gamma(\alpha_i/2)}\int_{(0,\infty)^n} \bigg(\prod_i t_i^{\alpha_i - 1}\bigg) e^{-\sum_i t_i}\,\rd \bb{t} \\
&= \prod_i \frac{2^{1 - \alpha_i/2}\Gamma(\alpha_i)}{\Gamma(\alpha_i/2)}
= \prod_i \frac{2^{\alpha_i/2}\Gamma((\alpha_i + 1)/2)}{\sqrt{\pi}}
= \prod_i \EE\left[|X_i|^{\alpha_i}\right].
\end{aligned}
\end{equation}
This proves the inequality in \eqref{eq:thm:strong.GPI}.

It remains to determine when equality holds under the assumption that every marginal variance is positive. If $R = I$, the coordinates are independent and equality is immediate. Suppose first that $R$ is positive definite and $R \neq I$. If equality held in \eqref{eq:inverse.square.kernel.bound}, then equality in the strictly convex Jensen inequality would hold for $\gamma_R$-almost every $\bb{y}$. Since $w_b(\bb{y}) > 0$ for every branch and
\[
\sum_b w_b(\bb{y})\bb{d}_b(\bb{y})^{\top}(I - R)\bb{d}_b(\bb{y}) = \tr\left((I - R)\diag(c_1,\ldots,c_n)\right) = 0,
\]
this would imply $\bb{d}_b(\bb{y})^{\top}(I - R)\bb{d}_b(\bb{y}) = 0$ for every branch $b$ and $\gamma_R$-almost every $\bb{y}$. By continuity, the equality would hold for every $\bb{y}$. For each branch, the map $\bb{y} \mapsto \bb{d}_b(\bb{y})$ is a diffeomorphism from $\R^n$ onto the corresponding orthant. Hence, the quadratic form $\bb{d}^{\top}(I - R)\bb{d}$ would vanish throughout an orthant. A polynomial that vanishes on a nonempty open set vanishes identically, which would imply $R = I$, a contradiction. Thus, $K_R(\bb{c}) > \exp(-\sum_i c_i)$ for every $\bb{c} \in (0,\infty)^n$. Since the weight $\prod_i c_i^{\alpha_i - 1}$ is strictly positive on $(0,\infty)^n$, the inequality in \eqref{eq:moment.recovery.comparison} is strict.

Finally, suppose that $R$ is singular. For $\bb{c} \in [0,\infty)^n$, define $K_R(\bb{c})$ by the same expectation, with zero coordinates of $\bb{c}$ contributing zero to the exponent. Since every marginal is standard normal and hence nonzero almost surely, dominated convergence shows that $K_R$ is continuous on $[0,\infty)^n$. Choose $i \neq j$ with $R_{ij} \neq 0$, which is possible because $R$ has unit diagonal and is singular. At a point $\bb{c}$ with $c_i, c_j > 0$ and all other coordinates equal to zero, the kernel reduces to the corresponding bivariate kernel. The bivariate inequality is strict when $|R_{ij}| < 1$ by the preceding positive-definite argument. When $|R_{ij}| = 1$, $X_j = R_{ij} X_i$ almost surely, so, with $a = (c_i^2 + c_j^2)^{1/2} > 0$, the identity $\EE[\exp(-a^2/(2Z^2))] = e^{-a}$ for $Z \sim \mathcal{N}(0,1)$ \citep[Eq.~(3.471.15)]{GradshteynRyzhik2000} gives the kernel value $e^{-a} > \exp(-c_i - c_j)$. Since the function $\bb{c} \mapsto K_R(\bb{c}) - \exp\left(-\sum_i c_i\right)$ is continuous on $[0,\infty)^n$, the kernel inequality remains strict throughout a relative neighborhood of this point $\bb{c}$. The intersection of this relative neighborhood with $(0,\infty)^n$ is a nonempty open set. Given that the weight $\prod_i c_i^{\alpha_i - 1}$ is strictly positive on $(0,\infty)^n$, integrating over this open set shows that the inequality in \eqref{eq:moment.recovery.comparison} is again strict. Therefore, equality implies $R = I$, which, for a Gaussian vector, is equivalent to independence of the coordinates.
\end{proof}

\section{Connection with recent proofs of strong polarization inequalities}\label{sec:connection}

The proof of Theorem~\ref{thm:strong.GPI} contains, as a specialization of \eqref{eq:determinantal.branch.identities}, the averaging identities used by \citet{MartinezOrtegaMoreno2026} to prove the strong polarization inequality and by \citet{GalicerOrtegaMorenoPinasco2026} to prove its weighted extension. Let $\bb{v}_1,\ldots,\bb{v}_n$ be linearly independent unit vectors, let $R = (\langle \bb{v}_i,\bb{v}_j\rangle)_{i,j}$ be their Gram matrix, and let $p_1,\ldots,p_n > 0$ satisfy $\sum_i p_i = 1$. In \eqref{eq:branch.equation}, set the vector parameter $\bb{y} = \bb{0}$ and take $c_i = p_i$. In this section, write $\bb{d}_b = \bb{d}_b(\bb{0})$ and $w_b = w_b(\bb{0})$. Define $\bb{u}_b = \sum_i d_{b,i}\bb{v}_i$ and $a_{b,i} = \langle \bb{v}_i,\bb{u}_b\rangle$. Then \eqref{eq:branch.equation} gives
\[
a_{b,i} = (R\bb{d}_b)_i = \frac{p_i}{d_{b,i}}, \qquad \|\bb{u}_b\|^2 = \bb{d}_b^{\top}R\bb{d}_b = \sum_i p_i = 1, \qquad \bb{u}_b = \sum_i \frac{p_i\bb{v}_i}{a_{b,i}}.
\]
The last identity is the normalized Lagrange multiplier equation on the unit sphere for the weighted product function $f(\bb{u}) = \prod_{i=1}^n |\langle \bb{v}_i,\bb{u}\rangle|^{p_i}$. Let $\mathcal{E}(f)$ denote the set of stationary points of $f$ on the unit sphere at which $f$ is nonzero, that is, the set of unit vectors $\bb{u}$ for which $f(\bb{u}) \neq 0$ and the normalized gradient condition
\[
\sum_{i=1}^n \frac{p_i\bb{v}_i}{\langle \bb{v}_i,\bb{u}\rangle} = \bb{u}
\]
holds. By the unique solvability of \eqref{eq:branch.equation} in each orthant established in Section~\ref{sec:main.result}, the map $b \mapsto \bb{u}_b$ identifies the branches bijectively with $\mathcal{E}(f)$. Writing $\bb{a}_b = (a_{b,1},\ldots,a_{b,n})^{\top}$, one has $\bb{a}_b = \bb{x}_b$ in the notation of the proof. Since $d_{b,i}^2/p_i = p_i/a_{b,i}^2$, Sylvester's determinant identity $\det(I + AB) = \det(I + BA)$ rewrites \eqref{eq:branch.weight} as
\[
w_b = \det\!\left(I + R\diag(p_i/a_{b,i}^2)\right)^{-1} = \det\!\left(I + \sum_i \frac{p_i\bb{v}_i \otimes \bb{v}_i}{a_{b,i}^2}\right)^{-1},
\]
where the second determinant is taken on the span of $\bb{v}_1,\ldots,\bb{v}_n$ and $(\bb{v}_i \otimes \bb{v}_i)\bb{x} = \langle \bb{v}_i,\bb{x}\rangle\bb{v}_i$. In these variables, \eqref{eq:determinantal.branch.identities} becomes
\[
\sum_b w_b = 1, \qquad \sum_b w_b\frac{p_i p_j}{a_{b,i}a_{b,j}} = p_i\delta_{ij}, \qquad i,j \in \{1,\ldots,n\},
\]
where $\delta_{ij}$ is the Kronecker delta. Taking the trace and using $\sum_b w_b = 1$ yields
\begin{equation}\label{eq:weighted.polarization.cancellation}
\sum_b\left(1 - \sum_i \frac{p_i^2}{\langle \bb{v}_i,\bb{u}_b\rangle^2}\right)w_b = 0.
\end{equation}

For $p_i = 1/n$, let $P(\bb{u}) = \prod_i \langle \bb{v}_i,\bb{u}\rangle$. Since $f(\bb{u}) = |P(\bb{u})|^{1/n}$, the set $\mathcal{E}(f)$ is precisely the set $\mathcal{E}(P)$ of \emph{extremal points} in the notation of \citet{MartinezOrtegaMoreno2026}. Under this identification, the preceding determinant formula gives $w_b = \mu(\bb{u}_b)$, where
\[
\mu(\bb{u}) = \det\!\left(I + \frac{1}{n}\sum_i \frac{\bb{v}_i \otimes \bb{v}_i}{\langle \bb{v}_i,\bb{u}\rangle^2}\right)^{-1}
\]
is the weight in \citet[Theorem~A]{MartinezOrtegaMoreno2026}. Multiplying \eqref{eq:weighted.polarization.cancellation} by $-n^2$ and rewriting the branch sum over $\mathcal{E}(P)$ gives their key averaging property:
\[
\sum_{\bb{u} \in \mathcal{E}(P)}\left(\sum_i \frac{1}{\langle \bb{v}_i,\bb{u}\rangle^2} - n^2\right)\mu(\bb{u}) = 0.
\]
Since every weight is positive, there exists $\bb{u} \in \mathcal{E}(P)$ such that
\begin{equation}\label{eq:strong.pol.conj}
\sum_{i=1}^n \frac{1}{\langle \bb{v}_i,\bb{u}\rangle^2} \leq n^2.
\end{equation}
This is the basis case of the {\it strong polarization inequality} proved by \citet{MartinezOrtegaMoreno2026}. Their limiting argument extends the averaging identity, and hence the inequality, from a basis to arbitrary nonparallel unit vectors. By the AM--GM inequality, the reciprocal-square bound \eqref{eq:strong.pol.conj} implies
\[
\sup_{\|\bb{u}\| = 1}\prod_{i=1}^n |\langle \bb{v}_i,\bb{u}\rangle| \geq n^{-n/2},
\]
which is the geometric formulation of Corollary~\ref{cor:real.linear.polarization} in Section~\ref{sec:consequences}. Their paper also constructs many extremal configurations for the strong polarization inequality arising from finite Coxeter systems and characterizes all extremal configurations for the real linear polarization problem.

\citet{GalicerOrtegaMorenoPinasco2026} extend this mechanism from equal weights to arbitrary positive weights. For rational weights, they choose $k_1,\ldots,k_n \in \N$, set $s = \sum_i k_i$, and write $\alpha_i = k_i/s$; their $\alpha_i$ is the quantity denoted here by $p_i$. Their Gram matrix $G$ is $R$, their matrix $A = G^{-1}$ is $R^{-1}$, their vector $\bb{x}$ on the sphere is $\bb{u}_b$, their coordinate vector $\bb{y} = (\langle \bb{v}_i,\bb{x}\rangle)_i$ is $\bb{a}_b = \bb{x}_b$, and their dual coordinate vector $\bb{z} = A\bb{y}$ is $\bb{d}_b$. Thus, their vector $\bb{y}$ is not the vector parameter $\bb{y}$ in \eqref{eq:branch.equation}. To avoid this collision, denote their vector here by $\bb{y}_{\mathrm{G}}$. Their polynomial system and weight are
\[
h_i(\bb{y}_{\mathrm{G}}) = y_{\mathrm{G},i}(A\bb{y}_{\mathrm{G}})_i - \alpha_i, \qquad
\mu_{\mathrm{G}}(\bb{y}_{\mathrm{G}}) = \det\!\left(A + \diag(\alpha_i/y_{\mathrm{G},i}^2)\right)^{-1};
\]
see \cite[pp.~4,~6]{GalicerOrtegaMorenoPinasco2026}. The simultaneous solutions $\bb{y}_{\mathrm{G}} \in \mathbb{C}^n$ of $h_i(\bb{y}_{\mathrm{G}}) = 0$, $i \in \{1,\ldots,n\}$, are precisely the vectors $\bb{a}_b$, one for each branch $b$. Also,
\[
A\bb{a}_b = \bb{d}_b, \qquad \mu_{\mathrm{G}}(\bb{a}_b) = \det(R)w_b.
\]
The Euler--Jacobi cancellation in \citet[p.~6]{GalicerOrtegaMorenoPinasco2026}, written with $\bb{y}_{\mathrm{G}}$ in place of their $\bb{y}$, is
\[
\sum_{h(\bb{y}_{\mathrm{G}})=\bb{0}}\left(s^2 - \sum_i \frac{k_i^2}{y_{\mathrm{G},i}^2}\right)\mu_{\mathrm{G}}(\bb{y}_{\mathrm{G}}) = 0.
\]
Because $k_i = sp_i$, this is exactly \eqref{eq:weighted.polarization.cancellation} multiplied by $s^2\det(R)$. Since all weights are positive, this cancellation selects a unit vector in the rational basis case. Passing from rational to arbitrary real weights and from bases to arbitrary systems of unit vectors, they obtain their {\it weighted strong polarization inequality}: there exists a unit vector $\bb{u}$ such that
\begin{equation}\label{eq:weighted.strong.pol.conj}
\sum_{i=1}^n \frac{p_i^2}{\langle \bb{v}_i,\bb{u}\rangle^2} \leq 1.
\end{equation}
The weighted AM--GM inequality then gives
\begin{equation}\label{eq:prod.ineq.2}
\sup_{\|\bb{u}\| = 1}\prod_{i=1}^n |\langle \bb{v}_i,\bb{u}\rangle|^{p_i} \geq \prod_{i=1}^n p_i^{p_i/2},
\end{equation}
which is the normalized geometric form of Corollary~\ref{cor:weighted.products.linear.functionals} in Section~\ref{sec:consequences}.

The proof of Theorem~\ref{thm:strong.GPI} diverges after the specialization corresponding here to $\bb{y} = \bb{0}$. The proofs of the strong polarization inequality and its weighted extension apply the Euler--Jacobi theorem to a polynomial system describing the stationary points and use the resulting scalar cancellation to select one of them. Here, the scalar cancellation is instead the trace of the full matrix identity \eqref{eq:determinantal.branch.identities}, obtained from the principal-minor expansion and the orthogonality of the matrix $U$. That matrix identity includes off-diagonal cancellations, holds for every $\bb{y} \in \R^n$ and every $c_1,\ldots,c_n > 0$, and does not require rational weights. All branches are then used simultaneously in the change-of-variables and Jensen step \eqref{eq:inverse.square.kernel.bound}, which yields the inverse-square Gaussian kernel bound and, after integral recovery of the moments, the strong GPI for arbitrary positive exponents.

\section{Consequences of Theorem~\ref{thm:strong.GPI}}\label{sec:consequences}

The consequences collected in this section include the exact real linear polarization constant when $\dim(\mathcal{H}) \geq n$ (Corollary~\ref{cor:real.linear.polarization}), a sharp weighted product inequality for real linear functionals (Corollary~\ref{cor:weighted.products.linear.functionals}), spherical product-moment bounds (Corollary~\ref{cor:spherical.product.moments}), a family of unequal-multiplicity hafnian inequalities for positive-semidefinite matrices together with a complete characterization of their equality cases (Corollary~\ref{cor:hafnian.inequalities}), logarithmic variance and covariance inequalities (Corollary~\ref{cor:logarithmic.covariance}), and a R\'enyi total-correlation certificate (Corollary~\ref{cor:Renyi.total.correlation}). The theorem also removes the positive-side GPI hypothesis from the mixed-sign lower bounds of \citet{LanOuimetSun2025} (Corollary~\ref{cor:mixed.sign.lower}).

\bigskip
As observed by \citet{MR2385646} and later emphasized by \citet{MR3425898}, the even-exponent GPI in \eqref{eq:original.GPI} implies the conjectured value of the real linear polarization constant. The conclusion is stated in the functional-analytic form introduced by \citet{MR1636556} and used in \eqref{eq:RLPC.conjecture}, together with its equivalent geometric formulation. Once Theorem~\ref{thm:strong.GPI} is available, the lower bound follows from the Riesz representation theorem, Gaussian radial decomposition, the standard Gaussian and chi moment formulas, and Stirling's formula, while optimality follows from the AM--GM inequality and Bessel's inequality.

\begin{corollary}[Real linear polarization constant]\label{cor:real.linear.polarization}
Let $\mathcal{H}$ be a real Hilbert space and let $f_1,\ldots,f_n$ be continuous real linear functionals on $\mathcal{H}$. Define $\left\|\prod_{i=1}^n f_i\right\| = \sup_{\|\bb{x}\| \leq 1}\prod_{i=1}^n |f_i(\bb{x})|$. Then
\[
\left\|\prod_{i=1}^n f_i\right\| \geq n^{-n/2}\prod_{i=1}^n \|f_i\|.
\]
The constant $n^{-n/2}$ is optimal whenever $\dim(\mathcal{H}) \geq n$. Equivalently, in this dimension range,
\[
\inf_{\substack{\bb{u}_1,\ldots,\bb{u}_n \in \mathcal{H} \\ \|\bb{u}_1\| = \cdots = \|\bb{u}_n\| = 1}}\sup_{\|\bb{x}\| = 1}\prod_{i=1}^n |\langle \bb{x},\bb{u}_i\rangle| = n^{-n/2}.
\]
\end{corollary}

\begin{proof}
The assertion is trivial if some $f_i = 0$. Otherwise, the Riesz representation theorem gives unit vectors $\bb{u}_i \in \mathcal{H}$ such that $f_i(\bb{x}) = \|f_i\|\langle \bb{x},\bb{u}_i\rangle$. Let $E$ be the span of $\bb{u}_1,\ldots,\bb{u}_n$, let $d = \dim(E)$, and let $\bb{G}$ be a standard Gaussian vector in $E$. Let
\[
M = \sup_{\substack{\bb{x} \in E \\ \|\bb{x}\| = 1}}\prod_{i=1}^n |\langle \bb{x},\bb{u}_i\rangle|.
\]
Orthogonal projection onto $E$ and homogeneity give
\[
\left\|\prod_{i=1}^n f_i\right\| = M\prod_{i=1}^n \|f_i\|.
\]
For every integer $m \geq 1$, Theorem~\ref{thm:strong.GPI} gives, with $Z \sim \mathcal{N}(0,1)$,
\[
M^{2m}\EE\left[\|\bb{G}\|^{2mn}\right] \geq \EE\left[\prod_{i=1}^n |\langle \bb{G},\bb{u}_i\rangle|^{2m}\right] \geq \left\{\EE\left[|Z|^{2m}\right]\right\}^n.
\]
Using the identities $\EE\left[|Z|^{2m}\right] = 2^m\Gamma(m + 1/2)/\sqrt{\pi}$ and $\EE\left[\|\bb{G}\|^{2mn}\right] = 2^{mn}\Gamma(mn + d/2)/\Gamma(d/2)$, applying Stirling's formula, and letting $m \to \infty$, one obtains $M \geq n^{-n/2}$, which proves the lower bound.

If $\dim(\mathcal{H}) \geq n$, choose $\bb{u}_1,\ldots,\bb{u}_n$ to be orthonormal and set $f_i(\bb{x}) = \langle \bb{x},\bb{u}_i\rangle$. For every unit vector $\bb{x}$, the AM--GM inequality and Bessel's inequality give
\[
\begin{aligned}
\prod_{i=1}^n |\langle \bb{x},\bb{u}_i\rangle|
&= \left(\prod_{i=1}^n |\langle \bb{x}, \bb{u}_i \rangle|^2\right)^{1/2} \\
&\leq \left(\frac{1}{n} \sum_{i=1}^n |\langle \bb{x}, \bb{u}_i \rangle|^2\right)^{n/2}
\leq \left(\frac{1}{n} \|\bb{x}\|^2\right)^{n/2}
= n^{-n/2},
\end{aligned}
\]
with equality at $\bb{x} = n^{-1/2}\sum_{i=1}^n \bb{u}_i$. This proves optimality. Since every $n$-tuple of unit vectors defines norm-one functionals in this way, the lower bound and the orthonormal example also give the equivalent infimum formula.
\end{proof}

\begin{remark}
As explained in Section~\ref{sec:connection}, the strong polarization inequality was proved recently by \citet{MartinezOrtegaMoreno2026} through an extremal-point averaging argument based on the Euler--Jacobi vanishing theorem. This reciprocal-square inequality implies the real linear polarization estimate \eqref{eq:RLPC.conjecture} by the AM--GM inequality, while their Theorem~B shows that the value $n^{-n/2}$ in the geometric formulation is attained only when $\bb{u}_1,\ldots,\bb{u}_n$ are orthonormal.
\end{remark}

Allowing the exponents in Theorem~\ref{thm:strong.GPI} to grow at different rates gives the following sharp weighted extension of Corollary~\ref{cor:real.linear.polarization}. After both sides of the inequality are raised to the power $1/V$, the resulting constant depends only on the normalized weights and is the natural multiplicative counterpart of the weighted AM--GM inequality.

\begin{corollary}[Weighted products of linear forms]\label{cor:weighted.products.linear.functionals}
Let $\mathcal{H}$ be a real Hilbert space, let $f_1,\ldots,f_n$ be continuous real linear functionals on $\mathcal{H}$, and let $\nu_1,\ldots,\nu_n > 0$. Let $V = \sum_{i=1}^n \nu_i$. Then
\[
\sup_{\|\bb{x}\| \leq 1}\prod_{i=1}^n |f_i(\bb{x})|^{\nu_i} \geq \prod_{i=1}^n\left(\frac{\nu_i}{V}\right)^{\nu_i/2}\prod_{i=1}^n \|f_i\|^{\nu_i}.
\]
The constant $\prod_i(\nu_i/V)^{\nu_i/2}$ is optimal whenever $\dim(\mathcal{H}) \geq n$.
\end{corollary}

\begin{proof}
Proceeding as in the proof of Corollary~\ref{cor:real.linear.polarization}, with $M = \sup_{\bb{x} \in E, \|\bb{x}\| = 1} \prod_{i=1}^n |\langle \bb{x},\bb{u}_i\rangle|^{\nu_i}$ and $S = \|\bb{G}\|$, one obtains from Theorem~\ref{thm:strong.GPI} that, for every integer $m \geq 1$,
\[
M^{2m}\EE[S^{2mV}] \geq \EE\left[\prod_{i=1}^n |\langle \bb{G},\bb{u}_i\rangle|^{2m\nu_i}\right] \geq \prod_{i=1}^n \EE[|Z|^{2m\nu_i}], \qquad Z \sim \mathcal{N}(0,1).
\]
Since
\[
\EE[|Z|^{2m\nu_i}] = \frac{2^{m\nu_i}\Gamma(m\nu_i + 1/2)}{\sqrt{\pi}}, \qquad \EE[S^{2mV}] = \frac{2^{mV}\Gamma(mV + \dim(E)/2)}{\Gamma(\dim(E)/2)},
\]
Stirling's formula yields
\[
\lim_{m \to \infty}\left\{\frac{\prod_{i=1}^n \EE[|Z|^{2m\nu_i}]}{\EE[S^{2mV}]}\right\}^{1/(2m)} = \prod_{i=1}^n\left(\frac{\nu_i}{V}\right)^{\nu_i/2},
\]
which proves the lower bound. If $\bb{u}_1,\ldots,\bb{u}_n$ are orthonormal and $p_i = \nu_i/V$, for every unit vector $\bb{x}$, the weighted AM--GM inequality and Bessel's inequality give
\[
\begin{aligned}
\prod_{i=1}^n |\langle \bb{x},\bb{u}_i\rangle|^{\nu_i}
&\leq \left(\prod_{i=1}^n p_i^{\nu_i}\right)^{1/2} \left(\sum_{i=1}^n |\langle \bb{x},\bb{u}_i\rangle|^2\right)^{V/2} \\
&\leq \left(\prod_{i=1}^n p_i^{\nu_i}\right)^{1/2} (\|\bb{x}\|^2)^{V/2} \\
&= \prod_{i=1}^n p_i^{\nu_i/2},
\end{aligned}
\]
with equality at $\bb{x} = \sum_{i=1}^n \sqrt{p_i}\bb{u}_i$. This proves optimality.
\end{proof}

\begin{remark}
The weighted inequality in Corollary~\ref{cor:weighted.products.linear.functionals} was proved recently, in an equivalent normalized form, by \citet[Theorem~2]{GalicerOrtegaMorenoPinasco2026}, who established the optimality of its constant in Section~2.3. Their proof of the lower bound uses their stronger weighted strong polarization inequality (Theorem~1), whereas the proof above derives the lower bound from Theorem~\ref{thm:strong.GPI} by allowing the exponents to grow at different rates and applying Gaussian radial decomposition and Stirling's formula. In both proofs, optimality follows from an orthonormal example. As \citet[pp.~2--3]{GalicerOrtegaMorenoPinasco2026} also observed, the sharp factor has a Shannon-entropy interpretation. Indeed, with $\bb{p} = (p_1,\ldots,p_n)$, where $p_i = \nu_i/V$, and $H(\bb{p}) = -\sum_i p_i\log(p_i)$,
\[
\prod_{i=1}^n \left(\frac{\nu_i}{V}\right)^{\nu_i/2} = \prod_{i=1}^n p_i^{V p_i/2} = \exp\left(\frac{V}{2}\sum_{i=1}^n p_i\log(p_i)\right) = \exp\left(-\frac{V}{2}H(\bb{p})\right).
\]
Thus, when $\nu_1 = \cdots = \nu_n = 1$, the constant $n^{-n/2}$ corresponds to the maximal entropy value $H(\bb{p}) = \log(n)$.
\end{remark}

Gaussian moments and spherical averages are linked by radial decomposition, a connection already exploited by \citet{MR2385646} for products of squared linear forms. Combining this decomposition with Theorem~\ref{thm:strong.GPI} gives a benchmark for unequal powers that is sharp when $d \geq n$: orthogonal directions then minimize the spherical product moment. This turns the probabilistic inequality into a geometric integral inequality on every Euclidean sphere.

\begin{corollary}[Spherical product moments]\label{cor:spherical.product.moments}
Let $E$ be a $d$-dimensional real Hilbert space, let $\bb{\Theta}$ be uniformly distributed on its unit sphere, let $\bb{u}_1,\ldots,\bb{u}_n \in E$ be unit vectors, and let $\alpha_1,\ldots,\alpha_n > 0$. Let $A = \sum_{i=1}^n \alpha_i$. Then
\[
\EE\left[\prod_{i=1}^n |\langle \bb{\Theta},\bb{u}_i\rangle|^{\alpha_i}\right] \geq \frac{\Gamma(d/2)}{\Gamma((d + A)/2)}\prod_{i=1}^n\frac{\Gamma((\alpha_i + 1)/2)}{\sqrt{\pi}}.
\]
Equality holds if and only if $\bb{u}_1,\ldots,\bb{u}_n$ are pairwise orthogonal. In particular, if $d \geq n$, the constant is optimal.
\end{corollary}

\begin{proof}
Let $\bb{G}$ be a standard Gaussian vector in $E$ and write $\bb{G} = S\bb{\Theta}$, where $S = \|\bb{G}\|$ is independent of $\bb{\Theta}$. Theorem~\ref{thm:strong.GPI} and the standard Gaussian and chi moment formulas give
\[
\EE[S^A]\EE\left[\prod_{i=1}^n |\langle \bb{\Theta},\bb{u}_i\rangle|^{\alpha_i}\right] = \EE\left[\prod_{i=1}^n |\langle \bb{G},\bb{u}_i\rangle|^{\alpha_i}\right] \geq \prod_{i=1}^n \frac{2^{\alpha_i/2}\Gamma((\alpha_i + 1)/2)}{\sqrt{\pi}},
\]
where $\EE[S^A] = 2^{A/2}\Gamma((d + A)/2)/\Gamma(d/2)$. This proves the inequality. By the equality statement in Theorem~\ref{thm:strong.GPI}, equality holds if and only if the Gaussian linear forms are independent, which is equivalent to the orthogonality of $\bb{u}_1,\ldots,\bb{u}_n$. Such a configuration exists when $d \geq n$.
\end{proof}

\begin{remark}
For $\alpha_1 = \cdots = \alpha_n = 2$, Corollary~\ref{cor:spherical.product.moments} recovers the spherical moment inequality of \citet[Theorem~2.2]{MR2385646}.
\end{remark}

Wick's formula identifies even Gaussian product moments with hafnians of repeated covariance matrices; see \citet[][Eq.~(4)]{MR2385646}. Frenkel's hafnian conjecture is an algebraic reformulation of the equal-even-exponent GPI. Theorem~\ref{thm:strong.GPI} proves Frenkel's hafnian conjecture and yields the following extension to unequal multiplicities.

\begin{corollary}[Hafnian inequalities]\label{cor:hafnian.inequalities}
Let $\Sigma = (\sigma_{ij})_{1 \leq i,j \leq n}$ be a positive-semidefinite matrix, let $\bb{m} = (m_1,\ldots,m_n)$, where $m_1,\ldots,m_n \in \{1,2,\ldots\}$, and let $M = \sum_{i=1}^n m_i$. Let $\Sigma[2\bb{m}]$ denote the $2M \times 2M$ matrix obtained by repeating each index $i$ exactly $2m_i$ times as both a row index and a column index. Then
\begin{equation}\label{eq:hafnian.ineq}
\haf\left(\Sigma[2\bb{m}]\right) \geq \prod_{i=1}^n (2m_i - 1)!!\,\sigma_{ii}^{m_i}.
\end{equation}
Equality holds if and only if $\Sigma$ has a zero row or is diagonal.
\end{corollary}

\begin{proof}
Let $\bb{X} \sim \mathcal{N}(\bb{0},\Sigma)$. Wick's formula and the one-dimensional Gaussian moment formula give
\[
\EE\left[\prod_{i=1}^n X_i^{2m_i}\right] = \haf\left(\Sigma[2\bb{m}]\right), \qquad \EE\left[X_i^{2m_i}\right] = (2m_i - 1)!!\,\sigma_{ii}^{m_i}.
\]
Apply Theorem~\ref{thm:strong.GPI} with $\alpha_i = 2m_i$. If $\sigma_{ii} = 0$ for some $i$, positive semidefiniteness forces the $i$th row to vanish, and both sides are zero. If every diagonal entry is positive, the equality statement in Theorem~\ref{thm:strong.GPI} shows that equality holds if and only if the Gaussian coordinates are independent, which is equivalent to $\Sigma$ being diagonal.
\end{proof}

\begin{remark}
View the $2m_i$ copies of index $i$ as vertices, assign weight $\sigma_{ij}$ to an edge joining a copy of $i$ to a copy of $j$, and define the weight of a matching as the product of its edge weights. Then $\haf(\Sigma[2\bb{m}])$ is the sum of the weights of all perfect matchings, whereas the right-hand side of \eqref{eq:hafnian.ineq} is the contribution from those matchings that pair copies only within the same index group.
\end{remark}

Logarithms of the absolute values of Gaussian variables arise, for example, in the linearized multivariate stochastic-variance model of \citet{HarveyRuizShephard1994} and in the wavelet-based regression estimator of the Hurst parameter proposed by \citet[Sections~2.1--2.2]{ParkPark2009}. The absolute-moment function $\alpha \mapsto \EE[|X|^{\alpha}]$ is the moment-generating function of $\log|X|$. Because Theorem~\ref{thm:strong.GPI} holds for every positive real exponent, the corresponding moment inequalities may be differentiated with respect to the exponent parameters at boundary points where one or more of them vanish, yielding the following sharp variance and covariance inequalities.

\begin{corollary}[Logarithmic variance and covariance inequalities]\label{cor:logarithmic.covariance}
Let $\bb{X} = (X_1,\ldots,X_n)$ be a centered Gaussian vector with $\Var(X_i) > 0$ for every $i$. Then, for every $a_1,\ldots,a_n \geq 0$,
\begin{equation}\label{eq:var.log.ineq}
\Var\left(\sum_{i=1}^n a_i\log|X_i|\right) \geq \frac{\pi^2}{8}\sum_{i=1}^n a_i^2.
\end{equation}
In particular, for every $i \neq j$,
\[
\Cov(\log|X_i|,\log|X_j|) \geq 0, \qquad \Cov(|X_i|^{\alpha},\log|X_j|) \geq 0, \qquad \alpha > 0.
\]
\end{corollary}

\begin{proof}
Rescaling the coordinates only adds constants to their logarithms and multiplies $|X_i|^{\alpha}$ by a positive constant, so it is enough to assume that every marginal is standard normal. Let $L = \sum_i a_i\log|X_i|$ and
\[
F(t) = \log\EE[e^{tL}] - \sum_{i=1}^n\log\EE[|X_i|^{ta_i}], \qquad t \geq 0.
\]
Restricting to the indices for which $a_i > 0$, one obtains from Theorem~\ref{thm:strong.GPI} that $F(t) \geq 0$ for $t > 0$. By H\"older's inequality, $L$ and each $\log|X_i|$ have finite exponential moments in a neighborhood of zero, so $F$ is twice differentiable at zero. Direct differentiation gives $F(0) = F'(0) = 0$, and therefore $F''(0) \geq 0$. For $Z \sim \mathcal{N}(0,1)$,
\[
\begin{aligned}
\Var(\log|Z|)
&= \frac{1}{4}\frac{\rd^2}{\rd s^2}\log\EE[|Z|^{2s}]\bigg|_{s = 0} \\
&= \frac{1}{4}\frac{\rd^2}{\rd s^2}\log\left\{\frac{2^s\Gamma(s + 1/2)}{\sqrt{\pi}}\right\}\bigg|_{s = 0}
= \frac{1}{4}\psi_1(1/2) = \frac{\pi^2}{8},
\end{aligned}
\]
where $\psi_1$ denotes the trigamma function and $\psi_1(1/2) = \pi^2/2$. Differentiating the definition of $F$ twice at zero and using the standard normality of the marginals, one obtains
\[
0 \leq F''(0) = \Var(L) - \sum_{i=1}^n a_i^2 \, \Var(\log|X_i|) = \Var(L) - \frac{\pi^2}{8}\sum_{i=1}^n a_i^2,
\]
which proves \eqref{eq:var.log.ineq}. Taking $a_i = a_j = 1$ and all other weights equal to zero gives the logarithmic covariance inequality. Finally, Theorem~\ref{thm:strong.GPI} applied to $(X_i,X_j)$ gives
\[
\EE[|X_i|^{\alpha}|X_j|^t] - \EE[|X_i|^{\alpha}]\EE[|X_j|^t] \geq 0, \qquad t > 0.
\]
The expression vanishes at $t = 0$, and its right derivative there is $\Cov(|X_i|^{\alpha},\log|X_j|)$, which is therefore nonnegative.
\end{proof}

\begin{remark}
For a standard bivariate Gaussian pair with correlation $\rho$ satisfying $|\rho| < 1$, differentiating the absolute-moment formula in \citet{MR0045347} gives the exact identity
\[
\Cov(\log|X_i|,\log|X_j|) = \frac{1}{2}\arcsin^2(\rho).
\]
For $|\rho| = 1$, the identity follows directly from $X_j = \pm X_i$ a.s. It follows that the logarithmic covariance inequality above is sharp, with equality if and only if the pair is independent.
\end{remark}

Let $P$ be the joint law of a random vector and let $Q$ be the product of its marginals. The total correlation introduced by \citet{MR0109755} is the Kullback--Leibler divergence $D_{\mathrm{KL}}(P\Vert Q)$ and measures departure from mutual independence. A natural order-$\gamma$ R\'enyi analogue of total correlation is obtained by replacing $D_{\mathrm{KL}}(P\Vert Q)$ with $D_{\gamma}(P\Vert Q)$ while keeping $P$ and $Q$ fixed; for background on R\'enyi divergence, see \citet{MR3225930}. At order $\gamma = 2$, when $P \ll Q$, this divergence is given by
\[
D_2(P\Vert Q) = \log\int (\rd P/\rd Q)^2\,\rd Q,
\]
and it is defined to be $+\infty$ otherwise. It is nonnegative and vanishes exactly when $P = Q$, equivalently, when the coordinates are mutually independent. The next result converts the excess in Theorem~\ref{thm:strong.GPI} into an explicit lower bound on $D_2(P\Vert Q)$. The bound uses a product moment as a test function and therefore requires no direct evaluation of the joint density ratio.

\begin{corollary}[R\'enyi total-correlation certificate]\label{cor:Renyi.total.correlation}
Let $R$ be a positive-definite correlation matrix, let $\bb{X} \sim \mathcal{N}(\bb{0},R)$, let $P$ be the law of $(|X_1|,\ldots,|X_n|)$, and let $Q$ be the product of its marginals. For $\bb{\alpha} = (\alpha_1,\ldots,\alpha_n) \in (0,\infty)^n$, let
\[
M_R(\bb{\alpha}) = \frac{\EE\left[\prod_{i=1}^n |X_i|^{\alpha_i}\right]}{\prod_{i=1}^n \EE[|Z|^{\alpha_i}]}, \qquad Z \sim \mathcal{N}(0,1).
\]
Then $M_R(\bb{\alpha}) \geq 1$ and
\[
D_2(P\Vert Q) \geq \log\left[1 + \frac{(M_R(\bb{\alpha}) - 1)^2}{\prod_{i=1}^n \left\{\EE[|Z|^{2\alpha_i}] / \EE[|Z|^{\alpha_i}]^2\right\} - 1}\right].
\]
\end{corollary}

\begin{proof}
The first assertion follows from Theorem~\ref{thm:strong.GPI}. If $D_2(P\Vert Q) = +\infty$, the claimed inequality is immediate. One may therefore assume that $D_2(P\Vert Q) < \infty$. Since $R$ is positive definite, both $P$ and $Q$ have strictly positive densities on $(0,\infty)^n$, and hence $P \sim Q$. Write $r = \rd P/\rd Q$, $f(\bb{x}) = \prod_i x_i^{\alpha_i}$, and $m = \EE_Q[f] = \prod_i \EE[|Z|^{\alpha_i}]$. Since $\EE_Q[r] = 1$, the Cauchy--Schwarz inequality gives
\begin{equation}\label{eq:to.rearrange}
\begin{aligned}
(M_R(\bb{\alpha}) - 1)^2 m^2
&= \left\{\EE_Q[(r - 1)f]\right\}^2 \\
&= \left\{\EE_Q[(r - 1)(f - m)]\right\}^2 \\
&\leq \EE_Q[(r - 1)^2]\Var_Q(f).
\end{aligned}
\end{equation}
Moreover,
\[
\EE_Q[(r - 1)^2] = \EE_Q[r^2] - 1 = \int \left(\frac{\rd P}{\rd Q}\right)^2 \rd Q - 1 = e^{D_2(P\Vert Q)} - 1,
\]
and
\[
\frac{\Var_Q(f)}{m^2} = \frac{\EE_Q[f^2]}{m^2} - 1 = \prod_{i=1}^n\frac{\EE[|Z|^{2\alpha_i}]}{\EE[|Z|^{\alpha_i}]^2} - 1.
\]
Since $\alpha_i > 0$ for every $i$, the random variable $f$ is nonconstant under $Q$, so $\Var_Q(f) > 0$. Rearranging \eqref{eq:to.rearrange} proves the claim.
\end{proof}

\begin{remark}
When $\alpha_i = 2m_i$ with $m_i \in \N$, the quantity $M_R(\bb{\alpha})$ is the hafnian in Corollary~\ref{cor:hafnian.inequalities} normalized by its diagonal benchmark. Corollary~\ref{cor:Renyi.total.correlation} therefore turns any strict hafnian excess into a quantitative lower bound on the R\'enyi total correlation of the Gaussian magnitudes.
\end{remark}

Mixed-sign moments combine inverse powers, which are singular near the origin, with positive powers, which emphasize the tails. Consequently, such moments are not controlled directly by a one-sided positive-moment inequality. \citet{LanOuimetSun2025} recently obtained a conditional-variance-adjusted lower bound contingent upon the corresponding lower-dimensional positive-side GPI. Theorem~\ref{thm:strong.GPI} removes that hypothesis in every dimension and makes the following estimate unconditional.

\begin{corollary}[Mixed-sign lower bounds]\label{cor:mixed.sign.lower}
Let $\bb{X} \sim \mathcal{N}(\bb{0},\Sigma)$, where $\Sigma$ is positive definite, and let $\emptyset \neq \mathcal{J} \subsetneq \{1,\ldots,n\}$. Let $\mathcal{J}^c = \{1,\ldots,n\} \setminus \mathcal{J}$, assume that $\beta_j \in [0,1)$ for $j \in \mathcal{J}$ and $\alpha_i \in (0,\infty)$ for $i \in \mathcal{J}^c$, and let $\Sigma_{\mathcal{I}\mathcal{I}'}$ denote the submatrix of $\Sigma$ with rows indexed by $\mathcal{I}$ and columns indexed by $\mathcal{I}'$. Define
\[
\Sigma/\Sigma_{\mathcal{J}\mathcal{J}} = \Sigma_{\mathcal{J}^c\mathcal{J}^c} - \Sigma_{\mathcal{J}^c\mathcal{J}}\Sigma_{\mathcal{J}\mathcal{J}}^{-1}\Sigma_{\mathcal{J}\mathcal{J}^c}.
\]
Then
\[
\begin{aligned}
\EE\left[\prod_{j \in \mathcal{J}} |X_j|^{-\beta_j}\prod_{i \in \mathcal{J}^c}|X_i|^{\alpha_i}\right]
&\geq \EE\left[\prod_{j \in \mathcal{J}} |X_j|^{-\beta_j}\right]\prod_{i \in \mathcal{J}^c}\left\{\frac{(\Sigma/\Sigma_{\mathcal{J}\mathcal{J}})_{ii}}{\Sigma_{ii}}\right\}^{\alpha_i/2}\EE\left[|X_i|^{\alpha_i}\right]\\
&\geq \prod_{j \in \mathcal{J}}\EE\left[|X_j|^{-\beta_j}\right]\prod_{i \in \mathcal{J}^c}\left\{\frac{(\Sigma/\Sigma_{\mathcal{J}\mathcal{J}})_{ii}}{\Sigma_{ii}}\right\}^{\alpha_i/2}\EE\left[|X_i|^{\alpha_i}\right].
\end{aligned}
\]
\end{corollary}

\begin{proof}
The first inequality follows from the lower-bound theorem of \citet{LanOuimetSun2025}, whose lower-dimensional positive-exponent GPI hypothesis is supplied by Theorem~\ref{thm:strong.GPI}. The second inequality follows from the nonpositive-exponent GPI of \citet{MR3278931}.
\end{proof}

\begin{remark}
For $i \in \mathcal{J}^c$, the Schur-complement entry $(\Sigma/\Sigma_{\mathcal{J}\mathcal{J}})_{ii}$ is the conditional variance $\Var(X_i \mid (X_j)_{j \in \mathcal{J}})$. The ratio inside each correction factor in Corollary~\ref{cor:mixed.sign.lower} is therefore the fraction of the marginal variance that remains after regression on the nonpositive-exponent block. It equals $1$ under block independence. In that case, the mixed moment on the left-hand side factors into the product of the two block moments, whereas the first lower bound becomes the moment of the nonpositive-exponent block times the product of the positive marginal moments.
\end{remark}

\section*{Statement of AI use}
\addcontentsline{toc}{section}{Statement of AI use}

An early version of the proof of Theorem~\ref{thm:strong.GPI}, which did not include the characterization of equality in terms of independence, was generated entirely by ChatGPT 5.6 Sol (Work with Max effort) in response to a prompt by Dylan Greaves. The prompt and output are available at \url{https://chatgpt.com/share/6a5ea69b-1648-83e8-80b1-014ae0b1003c}. Dylan Greaves formalized this early version of the proof in Lean using Codex; see \url{https://github.com/dylgre/gaussian-product-inequality}. Fr\'ed\'eric Ouimet reworked the proof of Theorem~\ref{thm:strong.GPI} into its present form to improve the exposition and add the characterization of equality in terms of independence. Corollaries~\ref{cor:weighted.products.linear.functionals}, \ref{cor:logarithmic.covariance}, and \ref{cor:Renyi.total.correlation} were identified and written using ChatGPT 5.6 Sol (Pro). Independently, Fr\'ed\'eric Ouimet derived Corollaries~\ref{cor:real.linear.polarization}, \ref{cor:spherical.product.moments}, and \ref{cor:hafnian.inequalities} based on his own knowledge of the problem. ChatGPT 5.6 Sol (Pro) was used throughout to assist with text editing.



\section*{Disclosure statement}
\addcontentsline{toc}{section}{Disclosure statement}

No potential conflict of interest was reported by the authors.

\section*{Acknowledgments}
\addcontentsline{toc}{section}{Acknowledgments}

The authors thank Donald Richards for his comments on an early version of this paper.

\section*{Funding}
\addcontentsline{toc}{section}{Funding}

Fr\'ed\'eric Ouimet is supported by the Natural Sciences and Engineering Research Council of Canada (NSERC) through Discovery Grant RGPIN-2026-04471 and Discovery Launch Supplement DGECR-2026-00449.

\section*{References}
\addcontentsline{toc}{section}{References}

\setlength{\bibsep}{0pt plus 0ex}

\bibliographystyle{plainnat}
\bibliography{bib}

\end{document}